\documentclass[a4paper,11pt,twoside]{article}
\usepackage[left=3cm, right=3cm, top=3cm, bottom=3cm]{geometry}
\usepackage{type1cm}
\usepackage[utf8]{inputenc}
\usepackage[T1]{fontenc}
\usepackage{aecompl}
\usepackage{ae}
\usepackage{textcomp}
\usepackage[english]{babel}
\usepackage{pdfsync}
\usepackage{tikz}
\usetikzlibrary{arrows,automata}
\usetikzlibrary{positioning}
\usepackage{pgffor}
\usepackage{dsfont}
\usepackage{capt-of}
\usepackage{amssymb}
\usepackage{varwidth}
\usepackage{multirow}
\usepackage{rotating}
\usepackage{booktabs}

\usepackage{subfig}
\usepackage[leftcaption]{sidecap}
\usepackage{caption}
\usepackage{hyperref}
\usepackage{hypbmsec}
\usepackage{amsmath,amsthm,amsfonts}
\usepackage{mathtools}
\usepackage{bbm}
\usepackage{enumerate}

\theoremstyle{plain}
\newtheorem{thm}{Theorem}[section]
\newtheorem{prop}[thm]{Proposition}

\newtheorem{lem}[thm]{Lemma}
\newtheorem{rem}[thm]{Remark}
\newtheorem{defn}[thm]{Definition}

\theoremstyle{definition}
\newtheorem{exam}[thm]{Example}

\renewcommand{\P}{\mathbb{P}}
\newcommand{\Z}{\mathbb{Z}}
\newcommand{\N}{\mathbb{N}}
\newcommand{\R}{\mathbb{R}}

\newcommand{\T}{\mathcal{T}}
\newcommand{\wiredT}[1][h]{\widetilde{\mathcal{T}}^{#1}}

\usepackage{cmll}
\newcommand{\up}{\shneg}
\newcommand{\down}{\shpos}

\newcommand{\G}{\mathsf{G}}
\renewcommand{\H}{\mathsf{H}}

\renewcommand{\d}{\mathrm{d}}

\newcommand{\abs}[1]{\lvert#1\rvert}
\newcommand{\indicator}[1]{\mathds{1}_{#1}}

\newcommand{\lcm}{\operatorname{lcm}}
\newcommand{\rec}{\mathsf{Rec}}
\newcommand{\rr}{\mathsf{RR}}
\newcommand{\sand}{\mathsf{SP}}
\renewcommand{\r}{\mathsf{r}}

\newenvironment{eqsystem}%
{\left[
\begin{aligned}}%
{\end{aligned}
\right.}
\usepackage{fancyhdr}
\hypersetup{%
  pdftitle = {The rotor-router group of directed covers of graphs},
  pdfauthor = {Wilfried Huss, Ecaterina Sava},
  pdfkeywords = {rotor-router walks, rotor-router group, sandpile group, directed cover}
}

\begin{document}
\title{The rotor-router group of directed covers of graphs}
\author{Wilfried Huss\footnote{Vienna University of Technology, Austria.}, 
Ecaterina Sava\footnote{Graz University of Technology, Austria. Research supported by the FWF program W 1230-N13.}}
\maketitle

\begin{abstract}
A rotor-router walk is a deterministic version of a random walk, in which the walker 
is routed to each of the neighbouring vertices in some fixed cyclic order.
We consider here directed covers of graphs (called also periodic trees)
and we study several quantities related to rotor-router walks on directed covers.
The quantities under consideration are: order of the rotor-router group,
order of the root element in the rotor-router group and the connection with random walks. 
\end{abstract}



%=============================================================================
\section{Introduction}

Given a finite connected and directed graph $\G$, one can construct
a labelled rooted tree $\T$ in the following way. The root vertex is labelled
with some $i\in\G$. Recursively if $x$ is a vertex in $\T$
with label $i\in\G$, then $x$ has $d_{ij}$ successors with label $j$. The tree $\T$ is called the 
\emph{directed cover} of $\G$. Random walks on directed covers of graphs
have been studied by \textsc{Takacs} \cite{Takacs97randomwalk}, 
\textsc{Nagnibeda and Woess} \cite{woess_nagnibeda}. 
On infinite graphs, their methods have been extended by \textsc{Gilch and Müller} \cite{gilch_mueller_covers}.

\emph{Rotor-router walks} are deterministic analogues to random walks, which
have been first introduced into the physics literature under the 
name \emph{Eulerian walks} by \textsc{Priezzhev, D.Dhar et al} \cite{PhysRevLett.77.5079}
as a model of \emph{self organized criticality}. In a rotor-router
walk on a graph, equip each vertex with an arrow (the rotor) pointing to one of the 
neighbours of the vertex.  A particle performing a rotor-router walk
carries out the following procedure at each step. First  it changes the 
rotor at its current position to point to the next neighbour, in a fixed order 
chosen at the beginning, and then moves to the neighbour the rotor is now pointing at. 
These deterministic walks have gained increased interest in the last years, and in many settings there is 
remarkable agreement between the behaviour of rotor-router walks and the expected behaviour of random walks. 
\textsc{Holroyd and Propp} \cite{holroyd_propp} proved that many quantities
associated to rotor-router walks such as normalized hitting frequencies, hitting
times and occupation frequencies, are concentrated around their expected values
for random walks. See also {\sc Cooper and Spencer} \cite{cooper_spencer_2006},
{\sc Doerr and Friedrich} \cite{doerr_friedrich_2007}, 
{\sc Angel and Holroyd} \cite{angel_holroyd_2011}, {\sc Kleber} \cite{kleber_goldbug},
and also \textsc{Cooper, Doerr et al.} \cite{cooper_doerr_friedrich_spencer_2006}. On the other hand,
rotor-router walks and random walks can also have striking differences. 
In questions concerning recurrence and transience of rotor-router walks on 
homogeneous trees, this has been proven by {\sc Landau and Levine} \cite{landau_levine_2009}.
For random initial configurations on homogeneous trees, see {\sc Angel and Holroyd}
\cite{angel_holroyd_2011}. In our note \cite{huss_sava_transience_dircovers} we have extended
their result to rotor-router walks on directed covers of finite graphs. 
Furthermore, one can use rotor-router walks for solving questions regarding the 
behaviour of random walks: for instance, in \cite{huss_sava_rotor} we have used a special rotor-router
process in order to determine the harmonic measure, that is, the exit distribution
of a random walk from a finite subset of a graph.

In this work, we continue the study of several quantities related to 
rotor-router walks such as the order of the rotor-router group and the order of the
root element in the rotor group on directed covers of finite graphs. On homogeneous trees, this was
done by \textsc{Levine}\cite{levine_sandpile_tree}.  
The remainder of the paper is structured as follows. In Section \ref{sec:preliminaries} we briefly 
review the definitions and basic properties of: graphs and trees, directed covers of graphs, rotor-router
walks, rotor-router and sandpile groups and the connection between them. We will
follow the notation from \cite{huss_sava_transience_dircovers}.

Section \ref{sec:order_rr_group} is dedicated to the study of the rotor-router group 
on finite pieces of directed covers $\T$ of finite graphs $\G$. In particular we study
the rotor-router group on balls with respect to the graph metric, that is on $\T^h = \{x\in\T: \d(r,x) \leq h\}$,
where $r$ is the root vertex of the graph and $\d(r,x)$ is the length of the shortest path from $r$ to $x$.
To the graph $\T^h$ we add a global sink vertex, which is connected to the root and all leaves.
By counting a certain family of rooted spanning forests, we give a recurrence formula in Theorem \ref{thm:recurent_rotor_group}
for the order of the rotor-router group on $\T^h$ in terms of the respective orders on the principal subbranches of $\T^h$.
Furthermore, in Theorem \ref{thm:asymptotics_rrorder} we show that the order grows doubly exponential in $h$.
The growth depends on the spectral radius of the adjacency matrix of $\G$.
Then, we consider the order of the root element in the rotor-router group, which can be defined as the number of particles
needed at the origin of $\T^h$, such that after performing a rotor-router walk and stopping the particles when they
hit the sink, we are back to the same rotor-router configuration we started from.
We describe in Theorem \ref{thm:root_element_order} a recursive way for finding
the order of the root element in terms of the respective orders of the subbranches.
A key tool in the proof of this result is the so-called \emph{explosion formula} introduced
by {\sc Angel and Holroyd} \cite{angel_holroyd_2011}. 

\section{Preliminaries}\label{sec:preliminaries}

\paragraph{Graphs and Trees.}
Let $\G=(V,E)$ be a locally finite and connected directed multigraph, with vertex set $V$
and edge set $E$. For sake of simplicity, we identify the graph $\G$ with its vertex 
set $V$, i.e., $i\in \G$ means $i\in V$. It will be clear from the context whether we are 
speaking about vertices or edges. If $(i,j)$ is an edge of $\G$, we write $i\sim_{\G}j$.
We write $\d(i,j)$ for the {\em graph distance}, that is, the length of the shortest path between
$i$ and $j$. A directed graph $\G$ is {\em strongly connected} if for every two vertices $i$ and $j$
there exists a directed  path from $i$ to $j$ and a directed path from $j$ to $i$.

Let $D=(d_{ij})_{i,j\in\G}$ be the \emph{adjacency matrix} of $\G$, where $d_{ij}$ is
the number of directed edges connecting $i$ to $j$.
We write $d_i$ for the sum of the entries in the $i$-th row of $D$, that is
$d_{i}=\sum_{j\in\G} d_{ij}$ is the {\em degree} of the vertex $i$ (or the number of the outgoing edges from $i$).
The adjacency matrix $D$ is {\em irreducible} if for every
pair of indices $i$ and $j$ there exists a natural number $n$ such that $d^{(n)}_{ij}>0$,
where $d^{(n)}_{ij}$ represents the $(i,j)$-entry of the matrix power $D^n$.
If $\G$ is strongly connected, its adjacency matrix is irreducible.

A {\em tree} $\T$ is a connected, cycle-free graph. A {\em rooted tree} is a tree with a distinguished vertex
$r$, called {\em the root}. For a vertex $x\in\T$, denote by $|x|$ the {\em height} of $x$, that is the graph distance from the root to $x$.
For any positive integer $h$, define the {\em truncated tree} $\T^h = \{x\in\T: |x| \leq h\}$ to be the subgraph of $\T$ induced by the
vertices at height smaller or equal to $h$. 

For a vertex $x\in \T\setminus\{r\}$, denote by $x^{(0)}$ its {\em ancestor}, that is
the unique neighbour of $x$ closer to the root $r$.
It will be convenient to attach an additional vertex $r^{(0)}$ to the root $r$, which will
be considered in the following as a sink vertex. Additionally we fix a planar embedding of $\T$ and enumerate
the neighbours of a vertex $x\in \T$ in counter-clockwise order $\big(x^{(0)}, x^{(1)}, \ldots, x^{(d_x-1)}\big)$
beginning with the ancestor.
We will call a vertex $y$ a \emph{descendant} of $x$, if $x$ lies on the unique shortest path from $y$
to the root $r$. A descendant of $x$, which is also a neighbour of $x$, will be called a \emph{child}.
A \emph{cone} $C_x$ rooted a $x$ is the subtree spanned by the descendants of $x$.
The {\em principal branches} of $\T$ are the cones rooted at the children of the root $r$.

The \emph{wired tree} $\wiredT$ of height $h$ is the multigraph obtained from $\T^h$ by
collapsing all leaves, i.e. all vertices $y\in \T$ with $\d(r,y) = h$, together with the ancestor $r^{(0)}$
of the root  to a single vertex $s$, the sink. We do not collapse multiple edges.

%=============================================================================
\paragraph{Directed Covers of Graphs.}

Let now $\G$ be a finite, directed and strongly connected multigraph with  adjacency matrix $D=(d_{ij})$. 
Let $m$ be the cardinality of the vertices of $\G$, and label the
vertices of $\G$ by $\{1,2,\ldots, m \}$.

The {\em directed cover} $\T$ of $\G$ is defined recursively as a rooted tree
$\T$ whose vertices are labelled by the vertex set $\{1,2,\ldots, m \}$ of $\G$.
The root $r$ of $\T$ is labelled with some $i\in\G$. Recursively, if $x$ is a vertex in $\T$ with
label $i\in \G$, then $x$ has $d_{ij}$ descendants with label $j$. We define
the {\em label function} $\tau:\T\to\G$ as the map that associates to each vertex in $\T$ its label 
in $\G$. The label $\tau(x)$ of a vertex $x$ will be also called the {\em type} of $x$. 
For a vertex $x\in\T$, we will not only need its type, but also the types of its children.
In order to keep track of the type of a vertex and the types of its children we
introduce the {\em generation function} $\chi = \left(\chi_i\right)_{i\in\G}$ with 
$\chi_i:\{1,\ldots,d_i\}\to\G$. For a vertex $x$ of type $i$,
$\chi_i(k)$ represents the type of the $k$-th child $x^{(k)}$ of $x$, i.e.,
\begin{equation*}
\text{if } \tau(x)=i \text{ then } \chi_i(k)=\tau(x^{(k)}), \text{ for } k=1,\ldots,d_i.
\end{equation*}
As the neighbours $\big(x^{(0)},\ldots, x^{(d_{\tau(x)})}\big)$ of any vertex $x$ are drawn in clockwise order,
the generation function $\chi$ also fixes the planar embedding of the tree and thus defines $\T$ uniquely
as a planted plane tree.

In order to distinguish between the two graphs
$\G$ and $\T$ we use the variables $i,j$ for vertices in $\G$ (and labels or types in $\T$) and $x,y$
for vertices of $\T$. The tree $\mathcal{T}$ constructed
in this way is called the \emph{directed cover} of $\G$. Such trees are also known as
\emph{periodic trees}, see {\sc Lyons} \cite{LP:book}, or \emph{trees with finitely many cone types} in 
{\sc Nagnibeda and Woess} \cite{woess_nagnibeda}. 
We say that the cone $C_x$ has {\em cone type} $\tau(x)$. Note that if $x,y\in\T$ have the
same label, that  is $\tau(x)=\tau(y)$, then the trees $C_x$ and $C_y$ are isomorphic as rooted trees.
Since $\G$ is a finite graph, the number of isomorphism classes of $C_x$, $x\in\T$ is finite. 

The graph $\G$ is called the \emph{base graph} or the \emph{generating graph} for the tree $\T$.
We write $\T_i$ for a tree with root $r$ of type $i$, that is $\tau(r)=i$.
In the following we give two basic examples of directed covers of graphs.

\begin{figure}
\centering
\begin{tikzpicture}[scale=0.65]
\begin{scope}[xshift=-3cm, yshift=1cm]
\node (chi) at (0,0) {
\begin{tabular}{lc|cc}
\multicolumn{2}{c}{} & \multicolumn{2}{l}{\scalebox{0.7}{$k\rightarrow$}} \\
\multicolumn{2}{c|}{$\chi_i(k)$} & 1 & 2 \\
\hline
\multirow{2}{*}{\scalebox{0.7}{\rotatebox{-90}{\rotatebox{90}{$i$} $\rightarrow$}}}
  & 1 & 2 & \\
  & 2 & 2 & 1
\end{tabular}
};
\end{scope}
\begin{scope}[xscale=0.8]
\coordinate (n0) at (4.65625,0);
\coordinate (n1) at (4.65625,1);
\coordinate (n3) at (1.625,2);
\coordinate (n22) at (1.625,3);
\coordinate (n24) at (0.5,4);
\coordinate (n30) at (0.5,5);
\coordinate (n32) at (0,6);
\coordinate (n31) at (1,6);
\coordinate (n23) at (2.75,4);
\coordinate (n26) at (2.0,5);
\coordinate (n29) at (2,6);
\coordinate (n25) at (3.5,5);
\coordinate (n28) at (3,6);
\coordinate (n27) at (4,6);
\coordinate (n2) at (7.6875,2);
\coordinate (n5) at (5.75,3);
\coordinate (n16) at (5.75,4);
\coordinate (n18) at (5.0,5);
\coordinate (n21) at (5,6);
\coordinate (n17) at (6.5,5);
\coordinate (n20) at (6,6);
\coordinate (n19) at (7,6);
\coordinate (n4) at (9.625,3);
\coordinate (n7) at (8.5,4);
\coordinate (n13) at (8.5,5);
\coordinate (n15) at (8,6);
\coordinate (n14) at (9,6);
\coordinate (n6) at (10.75,4);
\coordinate (n9) at (10.0,5);
\coordinate (n12) at (10,6);
\coordinate (n8) at (11.5,5);
\coordinate (n11) at (11,6);
\coordinate (n10) at (12,6);
\end{scope}

\draw (n1) -- (n0);
\draw (n3) -- (n1);
\draw (n22) -- (n3);
\draw (n24) -- (n22);
\draw (n30) -- (n24);
\draw (n32) -- (n30);
\draw (n31) -- (n30);
\draw (n23) -- (n22);
\draw (n26) -- (n23);
\draw (n29) -- (n26);
\draw (n25) -- (n23);
\draw (n28) -- (n25);
\draw (n27) -- (n25);
\draw (n2) -- (n1);
\draw (n5) -- (n2);
\draw (n16) -- (n5);
\draw (n18) -- (n16);
\draw (n21) -- (n18);
\draw (n17) -- (n16);
\draw (n20) -- (n17);
\draw (n19) -- (n17);
\draw (n4) -- (n2);
\draw (n7) -- (n4);
\draw (n13) -- (n7);
\draw (n15) -- (n13);
\draw (n14) -- (n13);
\draw (n6) -- (n4);
\draw (n9) -- (n6);
\draw (n12) -- (n9);
\draw (n8) -- (n6);
\draw (n11) -- (n8);
\draw (n10) -- (n8);

\fill[black] (n0) circle (2pt);
\fill[red] (n1) circle (2pt);
\fill[blue] (n3) circle (2pt);
\fill[red] (n22) circle (2pt);
\fill[blue] (n24) circle (2pt);
\fill[red] (n30) circle (2pt);
\fill[black] (n32) circle (2pt);
\fill[black] (n31) circle (2pt);
\fill[red] (n23) circle (2pt);
\fill[blue] (n26) circle (2pt);
\fill[black] (n29) circle (2pt);
\fill[red] (n25) circle (2pt);
\fill[black] (n28) circle (2pt);
\fill[black] (n27) circle (2pt);
\fill[red] (n2) circle (2pt);
\fill[blue] (n5) circle (2pt);
\fill[red] (n16) circle (2pt);
\fill[blue] (n18) circle (2pt);
\fill[black] (n21) circle (2pt);
\fill[red] (n17) circle (2pt);
\fill[black] (n20) circle (2pt);
\fill[black] (n19) circle (2pt);
\fill[red] (n4) circle (2pt);
\fill[blue] (n7) circle (2pt);
\fill[red] (n13) circle (2pt);
\fill[black] (n15) circle (2pt);
\fill[black] (n14) circle (2pt);
\fill[red] (n6) circle (2pt);
\fill[blue] (n9) circle (2pt);
\fill[black] (n12) circle (2pt);
\fill[red] (n8) circle (2pt);
\fill[black] (n11) circle (2pt);
\fill[black] (n10) circle (2pt);

\draw[dashed] (n10) -- (n32);
\node (sup) [right=0.125 of n10]  {$s^\up$};
\node (sdown) [right=0.125 of n0]  {$s_\down$};
\node (r) [right=0.125 of n1, yshift=-2]  {$r$};
\end{tikzpicture}
\caption{\label{fig:fib_tree}The wired Fibonacci tree $\wiredT[5]_2$ of height $5$ and root of type $2$.}
\end{figure}

\begin{exam}[{\bf Fibonacci tree}] The Fibonacci tree is the directed cover of the
graph $\G$ on two vertices $\{1,2\}$, with adjacency matrix 
%\begin{small}
\begin{equation*}
D=
\begin{pmatrix}
0 & 1 \\
1 & 1 
\end{pmatrix}
\end{equation*}
%\end{small}
It is a tree with two cone types: a vertex with label $1$ (of type $1$) in the tree has only one
child with label $2$ and a vertex of type $2$ has one child of type $1$ and 
one child of type $2$. In Figure \ref{fig:fib_tree} we have a wired Fibonacci tree with root type $2$
(the vertices of type $1$ are coloured in blue and those of type $2$ in red). 
The generation function $\chi$ is also given in the picture above.
\end{exam}

\begin{exam}[{\bf Bi-regular tree $\T$}] The bi-regular tree $\T$ with parameters $\alpha,\beta\in\N$
is the directed cover of the
graph $\G$ on two vertices  $\{1,2\}$, with adjacency matrix 
%\begin{small}
\begin{equation*}
D=
\begin{pmatrix}
0 & \alpha \\
\beta & 0 
\end{pmatrix}
\end{equation*}
%\end{small}
It is a tree with two cone types: every vertex in $\T$ with label $1$ has no child with label $1$
and $\alpha=d_{12}$ children with label $2$, and every vertex with label $2$ has $\beta=d_{21}$ children with label 
$1$ and no child with label $2$. Since in this case, on each level there are vertices of only one type,
the function $\chi$ has to be: $\chi_1(k)=2$, for $k=1,\ldots,\alpha$ and $\chi_2(k)=1$ for $k=1,\ldots,\beta$.
\end{exam}

\begin{figure}
\centering
\input{biregular_2_3.tex}
\caption{\label{fig:biregular_tree}The wired $(2,3)$-bi-regular tree $\wiredT[5]_1$ of height $5$ and root of type $1$.}
\end{figure}

%===============================================================================
\subsection{Rotor-Router Walks}
\label{subsec:rr_walks}

On a locally finite and connected graph $\G$, a \emph{rotor-router walk} is defined as follows. 
For each vertex $x\in\G$ fix a cyclic ordering $c(x)$ of its neighbours: $c(x)=\big(x^{(0)},x^{(1)},\ldots,x^{(d_x-1)}\big)$,
where  $x\sim_{\G} x^{(i)}$ for all $i=0,1,\ldots ,d_x-1$ and $d_x$ is the degree of $x$.
The ordering $c(x)$ is called the
\emph{rotor sequence} of $x$. A \emph{rotor configuration} is a function
$\rho: \G\to \G$, with $\rho(x) \sim_{\G} x$, for all $x\in \G$.
Hence $\rho$ assigns to every vertex one of its neighbours.
By abuse of notation, we write $\rho(x)=i$ if the rotor at $x$ points
to the neighbour $x^{(i)}$, with $i\in\{0,1,\ldots,d_x-1\}$.

A rotor-router walk is defined by the following rule.
Let $x$ be the current position of the particle, and $\rho(x) = i$ the state of
the rotor at $x$. In one step of the walk two things happen. First the position of the
rotor at $x$ is incremented to point to the next neighbour $x^{(i+1)}$ in the ordering $c(x)$, that is,
$\rho(x)$ is set to $i+1$ (with addition performed modulo $d_x$). Then the
particle moves to position $x^{(i+1)}$. The rotor-router walk is obtained by repeatedly
applying this rule.

Suppose now that $\G$ is a finite graph with $m$ vertices and fix a vertex $s$
in $\G$, which will represent the sink.

\paragraph*{Rotor-Router Group. }
Given a rotor configuration $\rho$ on $\G$, write $e_x(\rho)$ for the rotor
configuration resulting from starting a particle at $x$ and letting it perform
a rotor-router walk until it reaches the sink $s$. If a particle visits a vertex infinitely
often, it also visits all of its neighbours infinitely often; since $\G$ is connected
and finite, the particle eventually reaches the sink.

The set of edges $\big\{\big(x,\rho(x)\big):x\in \G\setminus\{s\}\big\}$ in a
rotor configuration forms a spanning 
subgraph of $\G$ in which every vertex except the sink $s$ has out-degree one. If this
subgraph contains no cycles, we call it an {\em oriented spanning tree} of $\G$.
Write $\rec(\G)$ for the set of oriented spanning trees of $\G$, which is also called
the set of recurrent configurations. It is easy to see that
if $\rho \in\rec(\G)$, then also $e_x(\rho)\in\rec(\G)$. For a proof, see 
{\sc Landau and Levine} \cite[Lemma 2.1]{landau_levine_2009}. Another
interesting property of the rotor configurations is that if $\rho_1,\rho_2\in\rec(\G)$
and $e_x(\rho_1)=e_x(\rho_2)$, then $\rho_1=\rho_2$; see once again \cite{landau_levine_2009}
for a proof. This means that the operation $e_x$ of adding a particle at $x$
and routing it to the sink acts invertibly on the set of recurrent rotor configurations.

The {\em rotor-router group $\rr(\G)$} of $\G$ is defined as the subgroup of the permutation
group of $\rec(\G)$ generated by $\big\{e_x:\: x\in \G\setminus\{s\}\big\}$.
For any two vertices $x$ and $y$,
the operators $e_x$ and $e_y$ commute. This is the so-called {\em abelian property}
of rotor-router walks. Hence the group $\rr(\G)$ is {\em abelian}.
Furthermore, $\rr(\G)$ acts transitively on $\rec(\G)$. More details can be found in
\cite{landau_levine_2009}.

\paragraph*{Sandpile Group.}
A {\em chip configuration} $\sigma$ on $\G$, also called a {\em sandpile} on $\G$, is a vector in 
$\Z^{m-1}$ of non-negative integers indexed by the non-sink vertices of $\G$, where $\sigma(x)$ represents the number
of chips at the vertex $x$. A chip configuration $\sigma$ is called {\em stable} if $\sigma(x)<d_{x}$,
for every non-sink vertex $x$. A vertex $x$ is {\em unstable} if $\sigma(x)\geq d_{x}$.
An unstable vertex may {\em topple}, by sending one chip to each neighbour. If $\sigma$ is 
not stable then one can show that by successively toppling unstable vertices, in finitely
many steps we arrive at a stable configuration $\sigma^\circ$. A stable chip configuration $\sigma$
is called {\em recurrent} if there exists a nonzero chip configuration $\delta$ such that 
$(\sigma+\delta)^\circ=\sigma$. The {\em sandpile group}
$\sand(\G)$ may be thought of as the set of recurrent chip configurations under the
operation $(\sigma+\delta)^\circ$ of addition followed by stabilization.
The order of the sandpile group $\sand(\G)$ is given by the determinant of the reduced Laplacian $\Delta'$,
compare with \cite[Lemma 2.8]{chip_rotor_2008}.
The {\em graph Laplacian} of $\G$ is the matrix $\Delta$ with entries
\begin{equation*}
\Delta_{ij}=
\begin{cases}
 d_{i}-d_{ii} & \text{ for } i=j,\\
 -d_{ij} & \text{ for } i\neq j.
\end{cases}
\end{equation*} 
The reduced Laplacian is 
obtained by deleting from the Laplacian matrix $\Delta$  of $\G$ the row and the column corresponding to the sink.
By the matrix-tree theorem, this determinant equals the number of {\em oriented spanning trees} of $\G$ rooted at the sink.

\begin{thm}\label{thm:rr_sp}
The rotor-router group $\rr(\G)$ for a connected finite graph $\G$ with a global sink is isomorphic to
its sandpile group $\sand(\G)$.
\end{thm}
The proof can be found in \cite[Theorem 2.5]{landau_levine_2009}.

Notation: two non-negative functions $f(n)$ and $g(n)$ have the same growth and we write 
$f(n)\asymp g(n)$ if there exist constants $c_1,c_2$ such that
$  c_1 g(n)\leq f(n)\leq c_2 g(n)$.

For the rest, we fix the following: 
\begin{itemize}
 \item $\G$ finite graph with $m$ vertices labelled by $\{1,2,\ldots ,m\}$.
 \item $\T_i$ directed cover of $\G$ with root $r$ of type $i\in\G$.
 \item $\wiredT_i$ the wired directed cover of height $h$ and root type $i$.
 \item $\tau(x)$ label function and $\chi_i(k)$ generation function.
\end{itemize}

%==========================================================================================
\section{Order of the Rotor-Router Group}\label{sec:order_rr_group}

In this section we want to describe the rotor-router group on directed covers of
finite graphs. For homogeneous trees, this was done in {\sc Levine}\cite{levine_sandpile_tree},
and the method used there fails when one considers non-homogeneous structures.
Because of the non homogeneity, our approach is also quite technical.
We will relate the rotor-router group of a wired directed cover of a graph with the
rotor-router group of its principal subbranches.
Denote by $t_i^h$ the {\em number of spanning trees} of $\wiredT_i$. 

From Section \ref{subsec:rr_walks}, 
in order to find an expression for the order $|\rr(\wiredT_i)|$ of the rotor-router
group of $\wiredT_i$, it is enough to count the number $t_i^h$ of the spanning trees,
since $|\rr(\wiredT_i)|=t_i^h$.
We will count spanning trees in terms of a class of spanning forests of the original tree $\T^h_i$,
which we define next.
\begin{defn}
We say that a spanning forest of a graph $\H$ is {\em rooted at a set of vertices} $S\subset \H$
if every connected component of the forest contains exactly one vertex of $S$.
\end{defn}
Denote by $\H/S$ the multigraph obtained from $\H$ by contracting $S$ into a single vertex $s$, while
keeping multiple edges that may have been created by this process.
\begin{lem}
\label{lem:trees_and_forests}
There is a bijection between the set of spanning trees of $\H/S$ and the set of spanning forests
of $\H$ which are rooted at $S$.
\end{lem}
\begin{proof}
Let $F$ be a spanning forest of $\H$ and denote by $F/S$ the spanning subgraph of $\H/S$ obtained by
contracting $S$ into a single vertex $s$. If $F$ has a connected component which does not
contain a vertex of $S$, its contraction $F/S$ is still not connected. On the other hand, if $F$
has a connected component which contains at least two vertices of $S$, collapsing $S$ into a single
vertex creates a cycle in $F/S$. Hence $F/S$ is a spanning tree of $\H/S$ if and only if $F$ is
a spanning forest rooted at $S$.

Let $E_S$ be the set of edges $e$ of $\H$ such that not both endpoints of $e$ are contained in $S$.
Consider the map $\psi: \H \to \H/S$, defined by
\begin{equation*}
\psi(v) = \begin{cases}
s,\quad &\text{for } v \in S \\
v,      &\text{otherwise},
\end{cases}
\end{equation*}
and its natural extension to the edge set $E$ of $\H$ which we also call $\psi$.
Then $\psi$ is a bijection between $E_S$ and the edges of $\H/S$.
Since every spanning forest $F$ of $\H$ rooted at $S$ is a subset of $E_S$,
the map $\psi$ naturally extends to a map $\widehat{\psi}: F\mapsto F/S$,
which maps $F$ onto its contraction. Hence $\widehat{\psi}$ is injective,
since the tree $F/S$ is fully defined by its edge set.

Now let $T$ be a spanning tree of $\H/S$, and define $F = \{\psi^{-1}(e):\: e\text{ is a edge of } T\}$.
Then $F$ is a forest from which we may obtain a spanning forest rooted at $S$ by adding every vertex of
$S$ which is not contained in a connected component of $F$ as a single vertex component. This construction
is again injective. Hence there exists a bijection between the set of spanning trees of $\H/S$ and the set
of spanning forests of $\H$ rooted on $S$.
\end{proof}

Back to the truncated tree $\T^h_i$, let us introduce {\em the down and up sinks}
\begin{equation}
\label{eq:s_down_s_up}
s_\down=\{r^{(0)}\} \text{ and } s^\up=\{x\in\T_i^h:|x|=h\},
\end{equation}
where $r^{(0)}$ is the ancestor of the root $r$, an additional vertex connected with the root 
$r$ of $\T_i^h$, and let
\begin{equation}
\label{eq:def_sink}
S=s_\down \cup s^\up.
\end{equation}
With the notation introduced above we have $\widetilde{\T}_i^h = \T^h_i / S$.
According to Lemma \ref{lem:trees_and_forests}, in order to compute the order of the
rotor-router group, one has to count the number of spanning forests of $\T_i^h$
rooted at $S$. We partition the spanning forests into two types. For all $i \in \G$
denote by
\begin{center}
$F^h_{i,\down} =$ \begin{minipage}[t]{0.6\linewidth}
the number of spanning forests of $\T_i^h$ rooted at $S$,
which contain an edge from $r$ to $s_\down$,
\end{minipage}

$F^h_{i,\up} =$ \begin{minipage}[t]{0.6\linewidth}
the number of spanning forests of $\T_i^h$ rooted at $S$,
which contain a path from $r$ to $s^\up$.
\end{minipage}
\end{center}
Using this notation we have that
\begin{equation}
 \big|\rr(\wiredT_i)\big|=t^h_i=F^h_{i,\down}+F^h_{i,\up}.
\end{equation}
and the order of the rotor-router group can be calculated recursively as follows.
\begin{thm}\label{thm:recurent_rotor_group}
The order of the rotor-router group $\rr(\wiredT_i)$ is given by 
$\big|\rr(\wiredT_i)\big|=F^h_{i,\down}+F^h_{i,\up}$, and 
the number of spanning forests $F^h_{i,\down}$ and $F^h_{i,\up}$ with $i\in\G$
can be calculated recursively as:
\begin{equation}\label{eq:span_forests}
\begin{eqsystem}
 F^h_{i,\down}  &= \prod_{j\in\G}\big(F_{j,\down}^{h-1}+F_{j,\up}^{h-1}\big)^{d_{ij}} \\
 F^h_{i,\up}  &= F^h_{i,\down} \sum_{j\in\G} \dfrac{d_{ij}F_{j,\up}^{h-1}}{F_{j,\down}^{h-1}+F_{j,\up}^{h-1}}.
\end{eqsystem}
\end{equation}
The initial values are $F_{i,\down}^1=1$ and $F_{i,\up}^1=d_i$, for all $i\in \G$.
\end{thm}
\begin{proof}
Let $\T_i^h$ be a tree with root $r$ of type $i\in G$, and of height $h$.
For $k=1,\ldots,d_i$, we denote by $\T_{\chi_i(k)}^{h-1}$ the cone of the $k$-th child $r^{(k)}$ of the 
root $r$ in $\T_i^h$. By the construction of the directed cover, the root $r^{(k)}$ of $\T_{\chi_i(k)}^{h-1}$ is of type $\chi_i(k)$.

Every spanning forest $\mathcal{F}$ that is counted in $F_{i,\down}^h$ contains the edge $(r,s_\down)$. In this case, whenever
the edge $(r^{(k)},r)$ is contained in $\mathcal{F}$, its restriction to $\T_{\chi_i(k)}^{h-1}$
does not contain a path from $r^{(k)}$ to the upper sink
of $\T_{\chi_i(k)}^{h-1}$. On the other hand, if the forest $\mathcal{F}$ does not contain the edge $(r^{(k)},r)$, then the 
restriction of the spanning forest to the first level cone must contain a path from $r^{(k)}$ to the upper sink. Otherwise the connected component containing
$r^{(k)}$ would not contain a sink vertex. As we can freely choose which of the vertices $\{(r^{(k)},r): k=1,\ldots,d_i\}$ are
part of the forest, we get
\begin{equation*}
F^h_{i,\down} = \sum_{w \in \{\down,\up\}^{d_i}} \prod_{k=1}^{d_i} F^{h-1}_{\chi_i(k),w_k}
\end{equation*}
where $w = (w_1,w_2,\dots,w_{d_i})$ is a word of length $d_i$ over the alphabet $\{\down,\up\}$.
In order to reduce the previous equation to the form given in \eqref{eq:span_forests}, take a
$j\in\G$ with $d_{ij}\neq 0$. This implies that for some $k=1,\ldots,d_i$
we have $\chi_i(k)=j$ (this happens exactly $d_{ij}$ times) and $F^h_{i,\down}$ can be factorized as
\begin{equation*}
 F^h_{i,\down}=\Big(F^{h-1}_{j,\down}+F^{h-1}_{j,\up}\Big)\sum_{w \in \{\down,\up\}^{d_i-1}} \prod_{k=1}^{d_i-1} F^{h-1}_{\chi_i(k),w_k}
\end{equation*}
This procedure can be repeated exactly $d_{ij}$ times, and we obtain
\begin{equation*}
 F^h_{i,\down}=\Big(F^{h-1}_{j,\down}+F^{h-1}_{j,\up}\Big)^{d_{ij}}\sum_{w \in \{\down,\up\}^{d_i-d_{ij}}} \prod_{k=1}^{d_i-d_{ij}} F^{h-1}_{\chi_i(k),w_k} 
\end{equation*}
By proceeding in the same way for all $j\in\G$ for which $d_{ij}\neq 0$, we get
\begin{equation*}
F^h_{i,\down}  = \prod_{j\in\G:\:d_{ij}\neq 0}\big(F_{j,\down}^{h-1}+F_{j,\up}^{h-1}\big)^{d_{ij}},
\end{equation*}
which can be obviously extended to all $j\in\G$, and we have proved the first part of \eqref{eq:span_forests}.

Consider now a spanning forest $\mathcal{F}$ with the root connected to the upper sink $s^\up$. Thus, whenever $m$ edges
from the set $\{(r^{(k)},r): k=1,\ldots,d_i\}$ are contained in $\mathcal{F}$, there are exactly $d_i - m$ possible ways
to connect $r$ to $s^{\up}$. Hence
\begin{equation}\label{eq:F_up}
F^h_{i,\up} = \sum_{w \in \{\down,\up\}^{d_i}} U(w)\prod_{k=1}^{d_i} F^{h-1}_{\chi_i(k),w_k},
\end{equation}
where $U(w) = \#\big\{k\in\{1,\ldots,d_i\}: w_k=\up\big\}$ represents the number of letters $\up$ in the word $w$.
Note that the number of spanning forests depends only on the adjacency matrix $D$ of the generating graph,
and not on the particular planar embedding $\chi$ of the tree. Therefore, to prove the equivalence of
\eqref{eq:F_up} and \eqref{eq:span_forests}, we can use induction on the number of children of
a vertex of a certain type. 

For the rest of the proof, fix $i\in\G$ and denote by $D^i$ the $i$-th row of the adjacency matrix  $D$ of $\G$.
If a vertex of type $i$ has only one child, say of type $l$, that is, $D^i=\big(\delta_j(l)\big)_{j\in\G}$,
then \eqref{eq:F_up} reduces to 
\begin{equation*}
 F^h_{i,\up}=\sum_{w\in\{\down,\up\}}U(w)F^{h-1}_{l,w}=F^{h-1}_{l,\up}
\end{equation*}
which proves the induction base. Assume now that for all possible choices of 
$F^{h-1}_{i,\up},F^{h-1}_{i,\down}\in\R$ and of the adjacency matrix $D$, the following holds
\begin{equation}\label{eq:equiv}
\sum_{w \in \{\down,\up\}^{d_i}} U(w)\prod_{k=1}^{d_i} F^{h-1}_{\chi_i(k),w_k}= 
\prod_{j\in\G}\big(F_{j,\down}^{h-1}+F_{j,\up}^{h-1}\big)^{d_{ij}}
\sum_{j\in\G} \dfrac{d_{ij}F_{j,\up}^{h-1}}{F_{j,\down}^{h-1}+F_{j,\up}^{h-1}},
\end{equation}
and we prove that it also holds if we increase one entry of $D$ by $1$.
Let us suppose that we increase the $d_{il}$-entry by $1$, for some fixed $l\in\G$.
This means that any vertex of type $i$ has an additional child of type $l$.
Let us denote by $\bar{D}$ this new matrix and by $\bar{d}_i$ and $\bar{\chi}_i$ the quantities
corresponding to the degree of a vertex and to the generation function for $\bar{D}$.
Due to the fact that the number of forests does not depend on the planar embedding, i.e.,
on the function $\chi_i$, we can put the new additional child of type $l$ on the rightmost position,
that is $\bar{\chi}_i(\bar{d}_i)=l$ and $\bar{\chi}_i(k)=\chi_i(k)$ for $k<\bar{d}_i$. Then we have
\begin{equation*}
\sum_{w \in \{\down,\up\}^{\bar{d}_i}} U(w)\prod_{k=1}^{\bar{d}_i} F^{h-1}_{\bar{\chi}_i(k),w_k}  =
\sum_{w \in \{\down,\up\}^{d_i+1}} U(w) \left(\prod_{k=1}^{d_i} F^{h-1}_{\chi_i(k),w_k}\right)F^{h-1}_{l,w_{d_i+1}}
\end{equation*}
which, factoring the word $w$ by its last letter, is equal to
\begin{align*}
 = \sum_{w \in \{\down,\up\}^{d_i}}\big(U(w)+1\big)\left(\prod_{k=1}^{d_i} F^{h-1}_{\chi_i(k),w_k}\right) F^{h-1}_{l,\up} 
 +\sum_{w \in \{\down,\up\}^{d_i}}U(w) \left(\prod_{k=1}^{d_i} F^{h-1}_{\chi_i(k),w_k}\right) F^{h-1}_{l,\down} \\
 =(F^{h-1}_{l,\up}+F^{h-1}_{l,\down})\sum_{w \in \{\down,\up\}^{d_i}}U(w)\prod_{k=1}^{d_i} F^{h-1}_{\chi_i(k),w_k}
 +F^{h-1}_{l,\up}\prod_{j\in\G}\big(F_{j,\down}^{h-1}+F_{j,\up}^{h-1}\big)^{d_{ij}}.
\end{align*}
Using the induction hypothesis \eqref{eq:equiv} this further equals
\begin{multline*}
\prod_{j\in\G}\big(F_{j,\down}^{h-1}+F_{j,\up}^{h-1}\big)^{d_{ij}+\delta_j(l)}
\sum_{j\in\G} \dfrac{d_{ij}F_{j,\up}^{h-1}}{F_{j,\down}^{h-1}+F_{j,\up}^{h-1}}
 +F^{h-1}_{l,\up}\prod_{j\in\G}\big(F_{j,\down}^{h-1}+F_{j,\up}^{h-1}\big)^{d_{ij}} \\
= \prod_{j\in\G}\big(F_{j,\down}^{h-1}+F_{j,\up}^{h-1}\big)^{d_{ij}+\delta_j(l)}
   \sum_{j\in\G} \dfrac{\big(d_{ij}+\delta_j(l)\big)F_{j,\up}^{h-1}}{F_{j,\down}^{h-1}+F_{j,\up}^{h-1}},
\end{multline*}
which proves the inductive step.
\end{proof}

Next, we evaluate asymptotically the behaviour of  $\big|\rr(\wiredT_i)\big|=F^h_{i,\down}+F^h_{i,\up}$,
for large values of $h$. 

\begin{lem}\label{lem:gamma_i}
For $i\in\G$ and $h\in\N$, the sequence $\gamma_i^h$ defined by
\begin{equation*}
\gamma_i^h=\frac{F^h_{i,\up}}{F^h_{i,\down}}
\end{equation*}
is convergent. If the spectral radius $\r(D) > 1$ then the limit $\varUpsilon_i=\lim_{h\to\infty}\gamma_i^h$ is positive.
\end{lem}
Recall that the \emph{spectral radius} $\r(A)$ of a square matrix $A$ is the maximal absolute value of all the
eigenvalues of $A$. To prove Lemma \ref{lem:gamma_i} we will need two additional results.
\begin{prop}
\label{prop:spectral_radius}
Let $A\in \R^{m\times m}$ be a non-negative matrix and let $\boldsymbol{v}\in\R^m$ be a positive vector. If there is a number
$b\geq 0$ such that $A\boldsymbol{v} \leq b\boldsymbol{v}$, then $\r(A) \leq b$. 
\end{prop}
For the proof see \textsc{Ding and Zhou} \cite[Proposition 2.2]{ding2009nonnegative}.
The next result is a simple application of Tarski's Fixed point Theorem for increasing functions on lattices, see
\textsc{Kennan} \cite[Theorem 3.3]{kennan_2001}.
\begin{thm}
\label{thm:concave_fixed_point}
Suppose $f$ is an increasing and strictly concave function
from $\R^m$ to $\R^m$ such that $f(\boldsymbol{0})\geq \boldsymbol{0}$, $f(\boldsymbol{a})>\boldsymbol{a}$ for some positive vector $\boldsymbol{a}$, and
$f(\boldsymbol{b})<\boldsymbol{b}$ for some vector $\boldsymbol{b}>\boldsymbol{a}$. Then $f$ has a unique positive fixed point. 
\end{thm}

\begin{proof}[Proof of Lemma \ref{lem:gamma_i}]
Let us fix $i\in\G$ for the rest of the proof.
The initial value is $\gamma^1_i=d_i$. Substituting in the definition of $\gamma_i^h$
the recurrence relations \eqref{eq:span_forests} for $F^h_{i,\up}$ and $F^h_{i,\down}$, we get
\begin{equation}
\label{eq:def_gamma}
\gamma_i^h=\sum_{j\in\G} \dfrac{d_{ij}F_{j,\up}^{h-1}}{F_{j,\down}^{h-1}+F_{j,\up}^{h-1}}
=\sum_{j\in\G} d_{ij} \dfrac{\gamma_j^{h-1}}{1+\gamma_j^{h-1}}, \text{ for } h\in \N.
\end{equation}
In order to prove that $\gamma_i^{h}$ is convergent in $h$, we show that $\gamma^h_i$ is monotone
and bounded. We first show by induction on $h$ that $\gamma_i^h$ is strictly decreasing.
We know that $\gamma^1_i=d_i$, and the induction basis $\gamma^2_i<\gamma^1_i$ follows from
\begin{equation*}
 \gamma^2_i=\sum_{j\in\G} d_{ij} \dfrac{\gamma_j^{1}}{1+\gamma_j^{1}}<\sum_{j\in\G} d_{ij}=d_i=\gamma^1_i. 
\end{equation*}
For the inductive step, we assume that $\gamma^h_i<\gamma^{h-1}_i$, 
and we prove that $\gamma^{h+1}_i<\gamma^h_i$. This follows from
\begin{equation*}
\gamma_i^{h+1}=\sum_{j\in\G} d_{ij}\frac{\gamma_j^{h}}{1+\gamma_j^{h}}<
\sum_{j\in\G}d_{ij}\frac{\gamma_j^{h-1}}{1+\gamma_j^{h-1}}=\gamma_i^h.
\end{equation*}
The sequence $\gamma^h_i$ is also bounded, i.e., $0\leq \gamma_i^h<d_i$, therefore it converges
to a limit. Denote by $\varUpsilon_i=\lim_{h\to\infty}\gamma_i^h$ and 
$\varUpsilon = (\varUpsilon_1,\ldots,\varUpsilon_m)\in \R^m_{\geq 0}$. Then
the limit vector $\varUpsilon$ is a solution of the fixed point equation
\begin{equation}
\label{eq:gamma_fixed_point}
 \varUpsilon_i=\sum_{j\in\G} d_{ij}\dfrac{\varUpsilon_j}{1+\varUpsilon_j}.
\end{equation}
In order to prove the positivity of the limit vector $\varUpsilon$ we apply Theorem \ref{thm:concave_fixed_point} to the function
$f = (f_1,\ldots,f_m):\R^m \to \R^m$ with
\begin{equation*}
f_i (\boldsymbol{x})= \sum_{j\in\G} d_{ij}\frac{x_j}{1+x_j} \text{ and } \boldsymbol{x}=(x_1,\ldots,x_m).
\end{equation*}
The function $f$ is obviously strictly concave and increasing. From the first part of the proof we have
$f(\boldsymbol{b}) < \boldsymbol{b}$ for $\boldsymbol{b}$ being the vector of initial values $\boldsymbol{b}=(d_1,\ldots,d_m)$.

Let now $\r(D) > 1$ and assume that for all positive vectors $\boldsymbol{a} = (a_1,\ldots,a_m)$ we have 
$f(\boldsymbol{a}) \leq \boldsymbol{a}$, which can be written as
\begin{equation}
\label{eq:Da_leq_a}
D\begin{pmatrix} 
  \frac{a_1}{1+a_1} \\
  \vdots \\
  \frac{a_m}{1+a_m}
 \end{pmatrix} \leq \boldsymbol{a}.
\end{equation}
Let $a^\star = \max \{a_i: i=1,\ldots,m\}$, then \eqref{eq:Da_leq_a} implies
$\frac{1}{1+a^\star} D \boldsymbol{a} \leq \boldsymbol{a}$, hence $D \boldsymbol{a} \leq (1+a^\star)\boldsymbol{a}$, and from Proposition \ref{prop:spectral_radius}
it follows that $\r(D) \leq 1 + a^\star$. Since $\boldsymbol{a}$ was arbitrary, $a^\star$ can be made arbitrary small, thus $\r(D) \leq 1$ which
is a contradiction to our assumption. Therefore there exists a positive vector $\boldsymbol{a}=(a_1,\ldots,a_m)$ such that $f(\boldsymbol{a}) > \boldsymbol{a}$. Theorem \ref{thm:concave_fixed_point} now ensures the existence of a unique positive fixed point of $f$, hence
the equation \eqref{eq:gamma_fixed_point} has a unique positive solution in addition to the trivial solution.

It remains to show that the limit vector cannot be zero.
Let $\bar{\gamma}^h_i$ be the sequence defined as $\gamma^h_i$ in \eqref{eq:def_gamma} with
initial values $\bar{\gamma}^1_i = a_i$. The vector $\boldsymbol{a}$ is such that $f(\boldsymbol{a}) > \boldsymbol{a}$.
It is easy to see that the sequence $\bar{\gamma}^h_i$ is increasing and $\gamma_i^{1} - \bar{\gamma}_i^{1}\geq 0$.
By induction on $h$, supposing $\gamma_i^{h} - \bar{\gamma}_i^{h}\geq 0$, we get
\begin{equation*}
\gamma_i^{h+1} - \bar{\gamma}_i^{h+1} =
\sum_{j\in\G} d_{ij}\left(\frac{\gamma_j^h}{1+\gamma_j^h} - \frac{\bar{\gamma}_j^h}{1+\bar{\gamma}_j^h}\right) 
= \sum_{j\in\G} d_{ij} \frac{\gamma_j^h - \bar{\gamma}_j^h}{(1+\gamma_j^h)(1+\bar{\gamma}_j^h)} \geq 0,
\end{equation*}
therefore $\gamma_i^{h} - \bar{\gamma}_i^{h}$ is non-negative. Hence $\gamma_i^h$ is bounded from below by $\bar{\gamma}_i^h$,
and the limit $\varUpsilon_i$ is positive.
\end{proof}

\begin{thm}
\label{thm:asymptotics_rrorder}
Let $\G$ be a finite, directed and strongly connected graph with vertex set $\{1,2,\ldots,m\}$
and adjacency matrix $D$. Let $\T_i$ be the directed cover of $\G$ with root $r$ of type $i$ and
$\wiredT_i$ be the wired tree defined as above. If $\r(D)>1$, then the order of the rotor-router group 
$\rr(\wiredT_i)$ grows doubly exponential:
\begin{equation}
 \big|\rr(\wiredT_i)\big|\asymp \exp\big\{\r(D)^h\big\},\quad \text{ for all } i\in\G,
\end{equation}
where $\r(D)$ is the spectral radius (Perron-Frobenius eigenvalue) of $D$. 
\end{thm}
\begin{proof}
Recall first that $ \big|\rr(\wiredT_i)\big|=F^h_{i,\down}+F^h_{i,\up}$.
In order to simplify the system of equations \eqref{eq:span_forests} let us make the following substitution:
for $i\in\G$ denote by $x_i^h=\log F^h_{i,\down}$ and $y_i^h=\log F^h_{i,\up}$
and apply the logarithm function to \eqref{eq:span_forests}. We get
\begin{equation}\label{eq:log_spanning_forests}
\begin{eqsystem}
x_i^h & = \sum_{j\in\G} d_{ij}x_j^{h-1}+\sum_{j\in\G} d_{ij}\log\big(1+\gamma_j^{h-1}\big) \\
y_i^h & =  x_i^h+\log\big(\gamma_i^h\big).
\end{eqsystem}
\end{equation}
Consider now the following vectors in $\R^m$:
\begin{equation*}
\boldsymbol{x}^h=  \begin{pmatrix}
x_1^h \\ \vdots \\x_m^h 
\end{pmatrix},
\quad \quad 
\boldsymbol{y}^h=\begin{pmatrix}
y_1^h \\ \vdots \\y_m^h 
\end{pmatrix},
\quad \quad 
\boldsymbol{\gamma}^h=\begin{pmatrix}
\gamma_1^h \\  \vdots \\ \gamma_m^h 
\end{pmatrix}.
\end{equation*}
Then \eqref{eq:log_spanning_forests} can be written in matrix form as
\begin{equation}\label{eq:vector_spanning_forests}
\begin{eqsystem}
\boldsymbol{x}^h &= D\Big(\boldsymbol{x}^{h-1}+\log\big(1+\boldsymbol\gamma^{h-1}\big)\Big)\\
\boldsymbol{y}^h &= \boldsymbol{x}^h+\log\big(\boldsymbol{\gamma}^h\big),
\end{eqsystem}
\end{equation}
where the function $\log$ in \eqref{eq:vector_spanning_forests} is applied
componentwise to the entries of the vectors $\boldsymbol\gamma^h$ and $(1+\boldsymbol\gamma^h)$ respectively.
The initial values are
$\boldsymbol{x}^1=(0,\ldots,0)^T$, $\boldsymbol{y}^1=(\log d_1, \ldots, \log d_m)^T$, and
$\boldsymbol\gamma^1=(d_1,\ldots, d_m)^T$. Then the solution of \eqref{eq:vector_spanning_forests}
is given by 
\begin{equation}
 \boldsymbol{x}^h=\sum_{k=1}^{h-1}D^k\log (1+\boldsymbol\gamma^{h-k}).
\end{equation}
From Lemma \ref{lem:gamma_i} the entries of the vector $\log\big(1+\boldsymbol\gamma^{h-k}\big)$ are bounded: for all
$i\in\{1,2,\ldots,m\}$ and $k=1,\ldots,h$
\begin{equation*}
0<\log(1+\varUpsilon_i)\leq\log (1+\gamma_i^{h-k})\leq\log(1+\gamma_i^1)=\log(1+d_i).
\end{equation*}
All $\varUpsilon_i$ and $d_i$ are positive. Write $c_i=\log(1+\varUpsilon_i)$
and $C_i=\log(1+d_i)$. Then
\begin{equation}\label{eq:Xh}
\left(\sum_{k=1}^{h-1}D^k\right)\boldsymbol{c} \leq \boldsymbol{x}^h \leq  \left(\sum_{k=1}^{h-1}D^k\right)\boldsymbol{C},
\end{equation}
where $\boldsymbol{c}=(c_1,\ldots,c_m)^T$ and 
$\boldsymbol{C}=(C_1,\ldots,C_m)^T$. The sum $\left(\sum_{k=1}^{h-1}D^k\right)$ behaves like $D^h$ for big values of $h$,
and the exponential growth rate of the matrix power $D^h$ as $h\to\infty$
is controlled by the eigenvalue of $D$ with the largest absolute value. 
Since $D$ is a non-negative and irreducible matrix, according to Perron-Frobenius
theorem for irreducible matrices, there exists a positive real number $\r(D)$
(the spectral radius of $D$), called the {\em Perron-Frobenius eigenvalue}
which is the eigenvalue of $D$ with the largest absolute value.
%The eigenvalue
%$\lambda$ is also simple, and it satisfies
%\begin{equation*}
%\min_i d_i\leq \lambda \leq \max_i d_i,
%\end{equation*}
%where $d_i=\sum_{j\in\G}d_{ij}$. Because of the irreducibility, $\min_i d_i\geq 1$, therefore
%the Perron-Frobenius eigenvalue of $D$ is greater that $1$. 
By \eqref{eq:Xh}, we can conclude that the behaviour of $\boldsymbol{x}^h$, for $h\to\infty$ is 
given by the greatest eigenvalue $\r(D)>1$ of $D$, i.e., for all $i$
\begin{equation*}
x_i^h\asymp \r(D)^h. 
\end{equation*}
Since $\boldsymbol{y}^h$ differs from $\boldsymbol{x}^h$ only by a bounded and decreasing quantity, see
\eqref{eq:vector_spanning_forests}, there exists 
$c>0$ such that $y_i^h\asymp \r(D)^h+c$.
Using $x_i^h=\log F^h_{i,\down}$ and $y_i^h=\log F^h_{i,\up}$, we get 
\begin{equation*}
F^h_{i,\down}\asymp \exp\{\r(D)^h\} \quad \text{ and }\quad F^h_{i,\up}\asymp \exp\{\r(D)^h\},
\end{equation*}
which implies
\begin{equation*}
 \big|\rr(\wiredT_i)\big|\asymp \exp\{\r(D)^h\},
\end{equation*}
and this proves the statement.
\end{proof}
\begin{exam}[The Fibonacci tree]
For the Fibonacci tree, the system of equations \eqref{eq:span_forests} can be written as 
\begin{equation*}
\begin{eqsystem}
F^h_{1,\down} &= F^{h-1}_{2,\down}+F^{h-1}_{2,\up}\\
F^h_{1,\up}   &= F^{h-1}_{2,\up}\\
\end{eqsystem}
\quad \quad
\begin{eqsystem}
F^h_{2,\down} &= \big(F^{h-1}_{1,\down}+F^{h-1}_{1,\up}\big)\big(F^{h-1}_{2,\down}+F^{h-1}_{2,\up}\big)\\
F^h_{2,\up}   &= F^{h-1}_{1,\up}\big(F^{h-1}_{2,\down}+F^{h-1}_{2,\up}\big)+F^{h-1}_{2,\down}\big(F^{h-1}_{1,\down}+F^{h-1}_{1,\up}\big)\\
\end{eqsystem}
\end{equation*}
The two sequences $\gamma^h_1$ and $\gamma^h_2$ are given recursively by
\begin{equation*}
\gamma^h_1=\dfrac{\gamma_2^{h-1}}{1+\gamma_2^{h-1}} \quad \text{ and } 
\quad \gamma_2^h= \dfrac{\gamma_1^{h-1}}{1+\gamma_1^{h-1}} +\dfrac{\gamma_2^{h-1}}{1+\gamma_2^{h-1}}.
\end{equation*}
The initial values are $\gamma^1_1=1$, $\gamma^1_2=2$ and the limit values are
\begin{equation*}
\varUpsilon_1=\lim_{h\to\infty}\gamma^h_1=\sqrt{2}-1 \quad \text{ and } 
\quad \varUpsilon_2=\lim_{h\to\infty}\gamma^h_2=\frac{\sqrt{2}}{2}.
\end{equation*}
The Perron-Frobenius eigenvalue of $D$ is $\frac{1+\sqrt{5}}{2}$, which is also related with
the Fibonacci numbers $\mathsf{F}_n$ by $\lim_{n\to\infty}\frac{\mathsf{F}_{n+1}}{\mathsf{F}_n}=\frac{1+\sqrt{5}}{2}$.
Finally, regarding the order of the rotor-router group $\rr(\wiredT_i)$ we have
\begin{equation*}
\big|\rr(\wiredT_i)\big|\asymp \exp\Bigg\{\Bigg(\frac{1+\sqrt{5}}{2}\Bigg)^h\Bigg\}. 
\end{equation*}
\end{exam}

\subsection{Order of the Root Element in the Rotor-Router Group}

\textsc{Levine} \cite{levine_sandpile_tree} computed the order of the root element
in the rotor-router group on homogeneous trees. His approach holds only for 
homogeneous trees, and it fails in our case. We describe here the 
order of the root element in terms of the respective orders on the principal subbranches.
Our method is yet another useful application of the {\em explosion formula} introduced in
\cite[Theorem 11]{angel_holroyd_2011}.

Like before, $\T_i$ is the directed cover of $\G$ with root $r$ of type $\tau(r)=i$. 
The {\em principal branches} of $\T_i$ are the subtrees $\T_{\chi_i(k)}$
rooted at the children $r^{(k)}$ of the root $r$ with type $\chi_i(k)=\tau(r^{(k)})$ and 
$k=1,2,\ldots, d_i$. Hence $\T_i$ has as principal branches $d_i=\sum_{j\in\G} d_{ij}$ subtrees.
Recall the definition of the truncated wired tree $\wiredT_i$, with the same root $r$ 
and sink $S=s_{\down}\cup s^{\up}$. Suppose now that we have one particle at the root $r$ of $\wiredT_i$, and we let it
perform a rotor-router walk until it hits the sink $S$, where it stops. Denote
by $\hat{r}_h$ the element of the rotor-router group $\rr(\wiredT_i)$ corresponding
to this process. Similarly, for all $k=1,\ldots,d_i$, denote by $\hat{r}^{(k)}_{h-1}$ the
element of the rotor-router group $\rr\big(\widetilde{\mathcal{T}}^{h-1}_{\chi_i(k)}\big)$
corresponding to one particle performing rotor-router walk on the principal subbranch
$\widetilde{\mathcal{T}}^{h-1}_{\chi_i(k)}$ starting at $r^{(k)}$.
Write $\langle\hat{r}_h \rangle$ for the cyclic subgroup of $\rr(\wiredT_i)$
generated by $\hat{r}_h$, and $\big\langle \big(\hat{r}^{(1)}_{h-1} ,\ldots ,\hat{r}^{(d_i)}_{h-1}\big) \big\rangle$
for the cyclic subgroup of $\bigoplus_{k=1}^{d_i}\rr\big(\widetilde{\mathcal{T}}^{h-1}_{\chi_i(k)}\big)$ 
generated by the element $\big(\hat{r}^{(1)}_{h-1} ,\ldots ,\hat{r}^{(d_i)}_{h-1}\big)$.
By \cite[Theorem 3.3]{levine_sandpile_tree} and from the isomorphism between the sandpile
and the rotor-router group of a tree, we have
\begin{equation}
 \rr(\wiredT_i)/\langle\hat{r}_h \rangle \simeq \bigoplus_{k=1}^{d_i} \widetilde{\mathcal{T}}^{h-1}_{\chi_i(k)}\Big/
\big\langle \big(\hat{r}^{(1)}_{h-1} ,\ldots ,\hat{r}^{(d_i)}_{h-1}\big) \big\rangle
\end{equation}
For simplicity of notation, we denote by $R^h_i$ the {\em order of the element $\hat{r}_h$}
in the rotor-router group $\rr(\wiredT_i)$, that is, the cardinality of the cyclic
group $\langle\hat{r}_h \rangle$, i.e. $ R_i^h=|\langle\hat{r}_h \rangle|$.
Recall the definition of the down and up sinks $s_{\down}$ and $s^{\up}$ respectively and let
\begin{center}
$S^h_{i,\down} =$ \begin{minipage}[t]{0.6\linewidth}
the number of particles stopped in $s_{\down}$ 
after  $R_i^h$  particles have been routed from $r$, 
\end{minipage}

$S^h_{i,\up} =$ \begin{minipage}[t]{0.6\linewidth}
the number of particles stopped in $s^{\up}$ 
after $R_i^h$ particles have been routed from $r$,
\end{minipage}
\end{center}
for $i\in\G$. Then
\begin{equation}\label{eq:sum_root_order}
 R_i^h=S^h_{i,\down}+S^h_{i,\up}.
\end{equation}
Write $R^{h-1}_{\chi_i(k)}$ for the number of particles started at the origin $r^{(k)}$
of the branch $\T^{h-1}_{\chi_i(k)}$.

\begin{thm}\label{thm:root_element_order}
The order $ R_i^h$ of the root element in the rotor-router group is given by
$R_i^h=S^h_{i,\down}+S^h_{i,\up}$ and $S^h_{i,\down},S^h_{i,\up}$ can be computed
recursively as follows:
\begin{equation}\label{eq:root_order_simplified}
\begin{eqsystem}
S^h_{i,\down} &= \lcm\Big(\Big\{R^{h-1}_{\chi_i(k)}: \text{ for }k=1,\ldots,d_i\Big\}\Big) \\
S^h_{i,\up}   &= S^h_{i,\down} \sum_{j\in\G}d_{ij}\dfrac{S^{h-1}_{j,\up}}{R^{h-1}_{j}},
\end{eqsystem}
\end{equation}
for all $i\in\G$. Here $\lcm$ represents the least common multiple.
The starting values are given by $S^1_{i,\down} = 1$ and $S^1_{i,\up}=d_i$,
for all $i\in\G$. 
\end{thm}

The proof of this result uses the \emph{explosion formula} introduced in \cite[Theorem 11]{angel_holroyd_2011}.
Since we will need this formula in the proof, we first adapt it here to our case.

\subsubsection{Explosion Formula}

The explosion formula, introduced in \cite[Theorem 11]{angel_holroyd_2011}, gives a
recursive formula for computing the number $E_n(\T_i,\rho)$ of particles which escape to infinity when
we start $n$ rotor-router walks at the root $r$ of a tree $\T_i$. The initial configuration of 
rotors on $\T_i$ is $\rho$.  For the tree $\T_i$ and principal branches $\T_{\chi_i(k)}$, 
rooted at the children $r^{(k)}$ of $r$, write $\rho_{\chi_i(k)}$
for the restriction of the rotor configuration $\rho$ on $\T_{\chi_i(k)}$, with $k=1,\ldots,d_i$.

Let us first introduce some notations, following mainly the notations from \cite[Theorem 11]{angel_holroyd_2011}. 
Let $e_n=e_n(\T_i,\rho)=\indicator{[\text{particle } n \text{ escapes to infinity}]}$, where $\indicator{}$
represents the indicator function, so that $E_n=\sum_{k=1}^ne_k$. Moreover, let
$e(\T_i,\rho)$ be the {\em escape sequence} $(e_1,e_2,\ldots)$.
Let $\N_+=\{1,2,\ldots\}$ and for sequences in $\N^{\N_+}$ we denote addition by
$(a_1,a_2,\ldots)+(b_1,b_2,\ldots):=(a_1+b_1,a_2+b_2,\ldots)$. Define the 
{\em shift operator } $\theta$ by
\begin{equation*}
\theta(a_1,a_2,\ldots):=(0,a_1,a_2,\ldots) 
\end{equation*}
and the {\em explosion operator} $\mathfrak{X}$ by 
\begin{equation*}
\mathfrak{X}(a_1,a_2,\ldots):=(1^{a_1},0,1^{a_2},0,\ldots)
\end{equation*}
where $1^k$ denotes a string of $k$ $1$s, or the empty string if $k=0$. We define a {\em majorization}
order $\preceq$ on sequences $(a_1,a_2,\ldots)\preceq(b_1,b_2,\ldots)$
if and only $\sum_1^n a_j\leq \sum_1^n b_j$ for all $n$. 
For $\mathbf{a}=(a_1,a_2,\ldots)$ denote by $\mathbf{a}_k = a_k$ its $k$-th element. For a finite sequence $\mathbf{a}$ the
length is denoted by $\abs{\mathbf{a}}$ and $\mathbf{a}^n$ is its $n$ times repetition.
We now adapt \cite[Theorem 11]{angel_holroyd_2011} for directed covers of finite graphs.
\begin{thm}[Explosion formula]
Let $\G$ be a finite graph with $m$ vertices and $\T_i$ its directed cover with root 
$r$ of type $i$. Fix a rotor configuration $\rho$ on $\T_i$. Then
\begin{equation*}
e(\T_i,\rho)=\mathfrak{X}\Bigg(\sum_{k=1}^{\rho(r)}\theta e\big(\T_{\chi_i(k)},\rho_{\chi_i(k)}\big)+
\sum_{k=\rho(r)+1}^{d_i}e\big(\T_{\chi_i(k)},\rho_{\chi_i(k)}\big)\Bigg).
\end{equation*}
\end{thm}

\begin{rem}
In case of a finite tree $\T^h_i$, we define the escape sequence $e$ such that $e_n(\T^h_i,\rho)=1$ if
the $n$-th particle reaches the upper sink $s^\up$ before reaching $s_\down$, and $0$ otherwise. Then
the explosion formula can be also written in the form
\begin{equation*}
e(\T^h_i,\rho)=\mathfrak{X}\Bigg(\sum_{k=1}^{\rho(r)}\theta e\big(\T^{h-1}_{\chi_i(k)},\rho_{\chi_i(k)}\big)+
\sum_{k=\rho(r)+1}^{d_i}e\big(\T^{h-1}_{\chi_i(k)},\rho_{\chi_i(k)}\big)\Bigg),
\end{equation*}
with initial values $e(\T^0_i,\rho) = (1,1,\ldots)$ for all $i\in\G$.
\end{rem}

\begin{proof}[Proof of Theorem \ref{thm:root_element_order}]
Fix an $i\in\G$.
% We first prove that
% \begin{equation}\label{eq:root_order}
% \begin{eqsystem}
% S^h_{i,\down} &= \lcm(R^{h-1}_{1},\ldots,R^{h-1}_{d_i}) \\
% S^h_{i,\up}   &= S^h_{i,\down} \sum_{k=1}^{d_i}\dfrac{S^{h-1}_{k,\up}}{R^{h-1}_{k}},
% \end{eqsystem}
% \end{equation}
% and then we show that this is equivalent with \eqref{eq:root_order_simplified}.
It is enough to consider only the zero rotor configuration $\underline{0}$, that is, at all
vertices the rotors point to the ancestors. Recall that $R^h_i$ is the order of the root
element in the rotor-router group of $\T^h_i$. Hence, after $R^h_i$ particles have performed a
rotor-router walk (stopped at the sink $S$) we are back to the configuration $\underline{0}$.
In other words, the escape sequence $e(\T^h_i, \underline{0})$ is periodic with  period 
$R^h_i$. Write $\mathbf{e}(\T^h_i)$ for the finite sequence consisting of the first full period
of  $e(\T^h_i, \underline{0})$. In order to adapt the explosion formula to the finite 
sequence $\mathbf{e}(\T^h_i)$, all the escape sequences involved need to be extended 
such that they have the same length. Let
\begin{equation*}
f^h_i(k) := \frac{\lcm\left(R^h_{\chi_i(1)},\ldots,R^h_{\chi_i(d_i)}\right)}{R^h_{\chi_i(k)}} .
%= \frac{S^h_{i,\down}}{R^h_{\chi_i(k)}},
\end{equation*}
We can now write the explosion formula for the first $R^h_i$ particles as
\begin{equation}
\label{eq:explosion_finite_sequence}
\mathbf{e}(\T^h_i) =
\mathfrak{X}\left(\mathbf{v}^{h-1}_i\right)
\quad
\text{ with }
\quad
\mathbf{v}^{h-1}_i = \sum_{k=1}^{d_i}\mathbf{e}\big(\T^{h-1}_{\chi_i(k)}\big)^{f^{h-1}_i(k)}.
%\mathfrak{X}\left(\sum_{k=1}^{d_i}\mathbf{e}\big(\T^{h-1}_{\chi_i(k)}\big)^{f^{h-1}_i(k)}\right)
\end{equation}
Hence $\mathbf{e}(\T^h_i)$ is a sequence of length $R^h_i = S^h_{i,\down} + S^h_{i,\up}$ consisting
of ''zeros`` and ''ones``, with $S^h_{i,\down}$ being the number of $0$'s and $S^h_{i,\up}$ being
the number of $1$'s. Since the number of $0$'s in a string $\mathfrak{X}(\mathbf{a})$ is equal to
$\abs{\mathbf{a}}$, we get 
\begin{equation*}
S^h_{i,\down} = \abs{\mathbf{v}^{h-1}_i} = \lcm\left(R^h_{\chi_i(1)},\ldots,R^h_{\chi_i(d_i)}\right).
\end{equation*}
On the other hand, because $S^h_{i,\up}$ equals the number of $1$'s in the escape sequence $\mathbf{e}(\T^h_i)$,
the explosion formula \eqref{eq:explosion_finite_sequence} gives
\begin{equation}\label{eq:Shi_expansion}
S^h_{i,\up} = \sum_{l=1}^{R^h_i} \mathbf{e}(\T^h_i)_l
= \sum_{l=1}^{R^h_i} \mathfrak{X}\left(\mathbf{v}^{h-1}_i\right)_l 
= \sum_{l=1}^{S^h_{i,\down}}   \left(\sum_{k=1}^{d_i}\mathbf{e}\big(\T^{h-1}_{\chi_i(k)}\big)^{f^{h-1}_i(k)}\right)_l.
\end{equation}
In the last equality we used the fact that for any finite sequence $\mathbf{a}$,
$\sum_{l=1}^{\abs{\mathbf{a}}} \mathbf{a}_l = \sum_{l=1}^{\abs{\mathfrak{X}(\mathbf{a})}} \mathfrak{X}(\mathbf{a})_l$
holds. By exchanging the order of summation in \eqref{eq:Shi_expansion}, we obtain
\begin{align*}
S^h_{i,\up} &=
\sum_{k=1}^{d_i}
   \sum_{l=1}^{S^h_{i,\down}} \left(\mathbf{e}\big(\T^{h-1}_{\chi_i(k)}\big)^{f^{h-1}_i(k)}\right)_l
 = \sum_{k=1}^{d_i} f^{h-1}_i(k) \sum_{l=1}^{R^{h-1}_{\chi_i(k)}} \mathbf{e}\Big(\T^{h-1}_{\chi_i(k)}\Big)_l \\
&= \sum_{k=1}^{d_i} f^{h-1}_i(k) S^{h-1}_{\chi_i(k),\up} 
 = S^h_{i,\down}\sum_{k=1}^{d_i} \frac{S^{h-1}_{\chi_i(k),\up}}{R^{h-1}_{\chi_i(k)}} 
= S^h_{i,\down}\sum_{j\in\G} d_{ij} \frac{S^{h-1}_{j,\up}}{R^{h-1}_j},
\end{align*}
which proves the theorem.
\end{proof}
We give now an alternative way of writing the system \eqref{eq:root_order_simplified}.
Using \eqref{eq:sum_root_order}, let us add the two equations in \eqref{eq:root_order_simplified}, 
and then divide the result through $R^h_i$. We get
\begin{equation}\label{eq:alternative_root_order}
\begin{eqsystem}
S^h_{i,\down} &= \lcm\Big(\Big\{R^{h-1}_{\chi_i(k)}: \text{ for }k=1,\ldots,d_i\Big\}\Big) \\
1             &= \dfrac{S^h_{i,\down}}{R^h_i}\Bigg(d_i+1-\sum_{j\in\G}d_{ij}\dfrac{S^{h-1}_{j,\down}}{R^{h-1}_{j}}\Bigg)
\end{eqsystem}
\end{equation}
Because of the $\lcm$ involved in \eqref{eq:root_order_simplified}, it is hard to derive 
asymptotics for $R_i^h$.

\paragraph*{Connection with the Random Walk.} 
Let $(X_t)$ be a simple random walk on $\T_i$
which starts at the root $r$ of $\T_i$ and define the {\em stopping times} $T_i$ (which depend on the root 
$r$ with type $\tau(r)=i$) as
\begin{equation*}
T_i=\inf\{t\geq 0: X_t\in s_{\down}\cup s^{\up}\}
\end{equation*}
and {\em hitting probabilities} $H_{i,\down}$ and $H_{i,\up}$ as
\begin{equation}\label{eq:hit_times}
H_{i,\down}=\P_{r}[X_{T_i}=s_{\down}] \quad \text{ and } \quad H_{i,\up}=\P_{r}[X_{T_i}\in s^{\up}]
\end{equation}
Write $H^h_{i,\down}$ and $H^h_{i,\up}$ for the corresponding probabilities on the truncated tree
$\T^h_i$. It is easy to see that $H_{i,\down}$ (or $H_{i,\up}$) can be expressed as the quotient 
between the number of particles $S^h_{i,\down}$ (or $S^h_{i,\up}$) that are routed in the
sink $s_{\down}$ (or $s^{\up}$) and the total number of particles $R_i^h$ started at the root,
that is
\begin{equation*}
H_{i,\down}=\dfrac{S^h_{i,\down}}{R^h_i} \quad \text{ and } \quad H_{i,\up}=\dfrac{S^h_{i,\up}}{R^h_i}.
\end{equation*}
One can also obtain that $H_i$ is the solution of the same equation \eqref{eq:alternative_root_order},
by factorizing the random walk $(X_t)$ with respect to the first step
\begin{equation}\label{eq:hit_prob}
 1=H^h_{i,\down}\Big(d_i+1-\sum_{k=1}^{d_{i}}H^{h-1}_{\chi_i(k),\down}\Big)=H^h_{i,\down}\Big(d_i+1-\sum_{j=1}^{m}d_{ij}H^{h-1}_{j,\down}\Big),
\end{equation}
with the initial value $H^1_{i,\down}=1/(d_i+1)$.

\begin{exam}[Fibonacci Tree] Consider the order of the root element in the
rotor-router group $\rr(\wiredT_i)$ of the wired Fibonacci tree $\wiredT_i$.
In this case the equation \eqref{eq:alternative_root_order} is given as:
\begin{equation*}
\begin{eqsystem}
S^h_{1,\down} & =R^{h-1}_2\\
R^h_1 & = S^h_{1,\down}\left(2-\dfrac{S^{h-1}_{2,\down}}{R^{h-1}_2}\right)
\end{eqsystem}
\quad \quad \quad \quad
\begin{eqsystem}
S^h_{2,\down} & = \lcm \big(R^{h-1}_1,R^{h-1}_2\big)\\
R^h_2 & = S^h_{2,\down}\left(3-\dfrac{S^{h-1}_{1,\down}}{R^{h-1}_1}-\dfrac{S^{h-1}_{2,\down}}{R^{h-1}_2}\right)
\end{eqsystem}
\end{equation*}
Computations suggest that on the Fibonacci tree, the sequences
$S^h_{i,\down}$ and $R_i^h$ are relatively prime. 
\end{exam}
\begin{exam}[Bi-regular tree]
For the $(\alpha,\beta)$-bi-regular tree the recurrence relation \eqref{eq:alternative_root_order}
for the order of the root element reduces to the simple form
\begin{align*}
R^h_1 & = R^{h-1}_2(\alpha + 1) - \alpha R^{h-2}_1 \\
R^h_2 & = R^{h-1}_1(\beta + 1) - \beta R^{h-2}_2,
\end{align*}
which has the explicit solution
\begin{equation*}
R^h_1 = \begin{cases}
\dfrac{(\alpha \beta)^k - 1}{\alpha\beta - 1} (\alpha + 1)\beta + 1, & \text{ for } h = 2k \\[1em]
\dfrac{(\alpha\beta)^{k+1} - 1}{\alpha\beta - 1}(\alpha + 1), & \text{ for } h = 2k + 1.
\end{cases}
\end{equation*}
For $R^h_2$ the solution is the same, with the roles of $\alpha$ and $\beta$ exchanged.
\end{exam}

%===================================================================================================
\bibliography{rrgroup_directed_covers}{}
\bibliographystyle{mypaperhep}

\vspace{1cm}
\begin{minipage}{0.45\textwidth}
Wilfried Huss\\
Vienna University of Technology\\
E-mail: \texttt{whuss@mail.tuwien.ac.at}\\
\texttt{http://www.math.tugraz.at/$\sim$huss}
\end{minipage}
\hfill
\begin{minipage}{0.45\textwidth}
Ecaterina Sava\\
Graz University of Technology\\
E-mail: \texttt{sava@tugraz.at}\\
\texttt{http://www.math.tugraz.at/$\sim$sava}
\end{minipage}

\end{document}